\documentclass[leqno]{amsart}
\usepackage{amsthm, amssymb, amsfonts}
\usepackage{lmodern}
\usepackage{color}
\usepackage{hyperref}
\usepackage[T5]{fontenc}
\usepackage{amscd,amssymb}
\usepackage[v2,cmtip]{xy}
\usepackage{mathrsfs}
\theoremstyle{plain}
\newtheorem{thm}{Theorem}[section] 
\newtheorem{prop}[thm]{Proposition}
\newtheorem{lem}[thm]{Lemma}

\newtheorem{corl}[thm]{Corollary}
\theoremstyle{definition}

\newtheorem{nota}[thm]{Notation}

\begin{document} 
\title[Steenrod operations on the modular invariants]{Steenrod operations on the modular\\ invariants}

\author{Nguy\~\ecircumflex n Sum}

	
\maketitle

\begin{abstract} 
In this paper, we compute the action of the mod $p$ Steenrod operations on the modular invariants of the linear groups with $p$ an odd prime number. 
\end{abstract}

\section*{Introduction}

Fix an odd prime $p$. Let $A_{p^n}$ be the alternating group on $p^n$ letters. Denote by $\Sigma_{p^n,p}$ a Sylow $p$-subgroup of $A_{p^n}$ and $E^n$ an elementary abelian $p$-group of rank $n$. Then we have the restriction homomorphisms
\begin{align*}&\mbox{Res}(E^n,\Sigma_{p^n,p}): H^*(B\Sigma_{p^n,p}) \longrightarrow H^*(BE^n),\\
&\mbox{Res}(E^n, A_{p^n}): H^*(BA_{p^n}) \longrightarrow H^*(BE^n),
\end{align*}
induced by the regular permutation representation $E^n \subset \Sigma_{p^n,p} \subset A_{p^n}$ of $E^n$ (see M\`ui
\cite{mui2}). Here and throughout the paper, we assume that the coefficients are taken in the prime field $\mathbb Z/p$. Using modular invariant theory of linear groups, M\`ui proved in \cite{mui1,mui2}
that
\begin{align*}&\mbox{ImRes}(E^n,\Sigma_{p^n,p}) = E(U_1,\ldots,U_n) \otimes P(V_1,\ldots,V_n),\\
	&\mbox{ImRes}(E^n, A_{p^n})= E(\tilde M_{n,0},\ldots , \tilde M_{n,n-1})\otimes P(\tilde L_n,Q_{n,1},\ldots, Q_{n,n-1}),
\end{align*}

Here and in what follows, $E(., \ldots, . )$ and $P(., \ldots , . )$ are the exterior and polynomial algebras over $\mathbb Z/p$ generated by the variables indicated. $\tilde L_n,\, Q_{,s}$ are the Dickson invariants of dimensions $p^n,\, 2(p^n - p^s)$, and $\tilde M_{n,s},, U_k,\, V_k$ are the M\`ui invariants of dimensions
$p^n-2p^s,\, p^{k-1},\, 2p^{k-1}$  respectively (see Section \ref{s2}).

Let $\mathcal A$ be the mod $p$ Steenrod algebra and let $\tau_s,\, \xi_i$ be the Milnor elements of dimensions $2p^s - 1,\, 2p^i - 2$ respectively in the dual algebra $\mathcal A_*$ of $\mathcal A$. In \cite{mil}, Milnor showed that, as an algebra
$$\mathcal A_* = E(\tau_0,\tau_1,\ldots)\otimes P(\xi_1,\xi_2,\ldots).$$
Then $\mathcal A_*$ has a basis consisting of all monomials $\tau_S\xi^R = \tau_{s_0}\ldots\tau_{s_k}\xi^{r_1}\ldots\xi^{r_m}$, with $S = (s_1,\ldots,s_k),\, 0 \leqslant s_1 < \ldots < s_k$, $R = (r_1, \ldots, r_m),\, r_i \geqslant 0$. Let $St^{S,R}\in \mathcal A$ denote the dual of $\tau_S\xi^R$  with respect to that basis. Then $\mathcal A$ has a basis consisting all operations $St^{S,R}$.
For $S = \emptyset,\, R = (r),\, St^{\emptyset,(r)}$ is nothing but the Steenrod operation $P^r$.

Since $H^*(BG),\ G = E^n,\, \Sigma_{p^n,p}$ or $A_{p^n}$, is an $\mathcal A$-module (see \cite[Chap. VI]{ste}) and the restriction homomorphisms are $\mathcal A$-linear, their images are $\mathcal A$-submodules of $H^*(BE^n)$.

The purpose of the paper is to study the module structures of $\mbox{ImRes}(E^n,\Sigma_{p^n,p})$ and $\mbox{ImRes}(E^n, A_{p^n})$ over the Steenrod algebra $\mathcal A$. More precisely, we prove a duality relation
between $St^{S,R}(\tilde M_{n,s}^\delta Q_{n,s}^{1-\delta})$ and $St^{S',R'}(U_{k+1}^\delta V_{k+1}^{1-\delta})$ for $\delta = 0, 1,\, \ell(R) = k$ and $\ell(R') = n$.
Here by the length of a sequence $T = (t_1,\ldots ,t_q)$ we mean the number $\ell(T) = q$. Using this relation we explicitly compute the action of the Steenrod operations $P^r$ on $U_{k+1},\, V_{k+1},\, M_{n,s}$ and $Q_{n,s}$.

The analogous results for $p = 2$ have been announced in \cite{sum}.
%\end{document}

The action of $P^r$ on $V_{k+1}$ and $Q_{n,s}$ has partially studied by Campbell \cite{cam}, Madsen \cite{mad}, Madsen-Milgram \cite{mmi}, Smith-Switzer \cite{ssw}, Wilkerson \cite{wil}. Eventually, this action was completly determined by Hung-Minh \cite{hum} and by Hai-Hung \cite{hhu}, Hung \cite{hun} for the case of the coefficient ring $\mathbb Z/2$.

The paper contains 3 sections. After recalling some needed information on the invariant theory, the Steenrod homomorphism $d_n^*P_n$ and the operations $St^{S,R}$ in Section \ref{s2}, we prove the duality theorem and its corollaries in Section \ref{s3}. Finally, Section \ref{s4} is an application of the duality theorem to determine the action of the Steenrod operations on the Dickson and M\`ui invariants.

\section*{Acknowledgement} 
The author expresses his warmest thanks to Professor Hu\`ynh M\`ui for generous help and inspiring guidance. He also thanks Professor Nguy\~\ecircumflex n H.V. H\uhorn ng for helpful suggestions which lead him to this paper.

\section{Preliminaries}\label{s2}

As is well-known $H^*(BE^n) = E(x_1,\ldots,x_n) \otimes P(y_1,\ldots,y_n)$ where $\dim x_i = 1,\,
y_i = \beta x_i$ with $\beta$ the Bockstein homomorphism. Following Dickson \cite{dic} and M\`ui \cite{mui1}, we
define
\begin{align*}
	&[e_1, \ldots,e_k] = \det (y_i^{p^{e_j}}),\\
	&[1;e_{2},  \ldots,  e_k]  =  
	\begin{vmatrix} x_1&\cdots &x_k\\
	y_1^{p^{e_{2}}}&\cdots &y_k^{p^{e_{2}}}\\
	\vdots&\cdots  &\vdots\\
	y_1^{p^{e_k}} & \cdots & y_k^{p^{e_k}}
	\end{vmatrix} . 
\end{align*}
for every sequence of non-negative integers $(e_1 , \ldots , e_k),\, 1 \leqslant k \leqslant n$. We set
\begin{align*}
	&L_{k,s} = [0, \ldots, \hat s,\ldots,k],\, L_k = L_{k,k} = [0,\ldots,k-1],\, L_0 = 1,\\
	&M_{k,s} = [1;0, \ldots , \hat s,\ldots, k-1],\, 0\leqslant s < k \leqslant n.
\end{align*}
Then $\tilde L_n,\, Q_{n,s},\, \tilde M_{n,s},\, U_k,\, V_k$ are defined by
\begin{align*}
	&\tilde L_{n} = L_n^h,\, h = (p-1)/2,\, Q_{n,s} = L_{n,s}/L_n,\, 0\leqslant s \leqslant n,\\
	&\tilde M_{k,s} = M_{n,s}L_n^{h-1},\, U_k = M_{k,k-1}L_{k-1}^{h-1},\, V_k = L_k/L_{k-1}, \, 1 \leqslant k \leqslant n.
\end{align*}
Note that $Q_{n,0} = \tilde L_n^2$, $Q_{n,n} = 1$ for any $n > 0$.

Let X be a topological space. Then we have the Steenrod power map
$$P_n: H^q(X)\, \longrightarrow  H^{p^nq}
(EA_{p^n}\underset{A_{p^n}}\times X^{p^n}), $$
which sends $u$ to $1 \otimes u^{p^n}$ at the cochain level (see \cite[Chap. VII]{ste}). We also have the diagonal homomorphism
$$d_n^* : H^*(EA_{p^n}\underset{A_{p^n}}\times  X^{p^n}) \longrightarrow \,   H^*(BE^n )\otimes H^*(X) $$
induced by the diagonal map of $X$, the inclusion $E^n \subset A_{p^n}$ and the K\"unneth formula.

$d_n^*P_n$ has the following fundamental properties.

\begin{prop}[M\`ui \cite{mui1,mui2}]\label{md11}\

\medskip
{\rm (i)} $d_n^*P_n$ is natural monomorphism preserving
cup product up to a sign, more precisely
$$d^*_nP_n(uv)   = (-1)^{nhqr}d^*_nP_nud^*_nP_nv\ ,  $$
where $q = \dim u,\, r = \dim v,\, h = (p-1)/2.$

\medskip 
	{\rm (ii)} $d^*_nP_n = d^*_{n-s}P_{n-s}d^*_sP_s\ ,\ 0 \leqslant s \leqslant n. $
	
\medskip
	{\rm (iii)} For $H^*(E^1) = E(x)\otimes P(y)$, we have
\begin{align*}
&d_n^*P_nx = (-h!)^nU_{n+1} = (h!)^n\Big(\tilde L_nx + \sum_{s=0}^{n-1}(-1)^{s+1}\tilde M_{n,s}y^{p^s}\Big),\\
&d_n^*P_ny = V_{n+1} = (-1)^n\sum_{s=0}^n(-1)^sQ_{n,s}y^{p^s}.
\end{align*}
where $U_{n+1} = U_{n+1}(x_1, \ldots,x_n,x,y_1,\ldots,y_n,y)$, $V_{n+1} = V_{n+1}(y_1,\ldots,y_n,y)$.
\end{prop}

The following is a description of $d_n^*P_n$ in terms of modular invariants and cohomology operations.

\begin{thm}[{M\`ui \cite[1.3]{mui2}}]\label{dl12} Let $ z \in H^q(X)$,\, $\mu(q)=(h!)^q(-1)^{hq(q-1)/2}$. We then have 
	$$d_n^*P_nz\,  =\, \mu (q)^n \sum_{S,R}
	(-1)^{r(S,R)}\tilde  M_{n,s_1} \ldots\tilde  M_{n,s_k}\tilde L_n^{r_0}Q_{n,
		1}^{r_1}\ldots Q_{n,n-1}^{r_{n-1}}\ \otimes\ St^{S,R}z\ .$$
Here the sum rwns over all $(S,R)$ with $S = (s_1, \ldots , s_k)$, $0\leqslant s_1 < \ldots < s_k$, $ R = (r_1,\ldots, r_n),\, r_i \geqslant 0$,  $r_0 = q - k - 2(r_1 +\ldots\ +r_n)\geqslant 0,\ r(S,R) = k + s_1 + \ldots\ + s_k + r_1 + 2r_2 +\ \ldots\ + nr_n\ .$
\end{thm}
\section{The duality theorem}\label{s3}

Let $\tilde m_{m,s},\, \tilde \ell_m,\, q_{m,s},\, m = n$ or $k$,
(resp. $u_{k+1},\,  v_{k+1}$) be the dual of
$\tilde M_{m,s},\, \tilde L_m,\, Q_{m,s}$ (resp. $U_{k+1},\, V_{k+1}$) in
$$ E(\tilde M_{m,0},\ldots,\tilde M_{m,m-1}) \otimes
P(\tilde L_m,Q_{m,1}\ldots, Q_{m,m-1})$$
(resp. $E(U_{k+1})\otimes P(V_{k+1})$) with respect to the basis consisting of all monomials 
\begin{align*}\tilde M_S\tilde Q^H &=  \tilde M_{m,s_1} \ldots \tilde M_{m,s_k} \tilde L_m^{h_0}Q_{m,1}^{h_1}\ldots Q_{m,m-1}^{h_{m-1}},
\end{align*}
with $S = (s_1,\ldots, s_k),\,   0 \leqslant s_1  < \ldots < s_k, \, H = (h_0,\ldots,h_{m-1}),\, h_i \geqslant 0 ,$ (resp. $ U_{k+1}^eV_{k+1}^j;\,
e = 0,  1,\ j \geqslant 0$). Let $\Gamma(\tilde\ell_m,q_{m,1},\ldots,
q_{m,m-1})$ (resp.  $\Gamma(v_{k+1})$)  be the divided polynomial algebra with divided power $\gamma_i,\, i \geqslant 0$ generated by $\tilde\ell_m,\  q_{m,   1},  \ldots$, $
q_{m, m-1}$\  (resp.  $v_{k+1}$).  We set
$$\tilde m_S\tilde q_H =  \tilde m_{m,s_1} \ldots \tilde m_{m, s_k}\gamma_{h_0}(\tilde   \ell_m)\gamma_{h_1}(q_{m,1})\ldots
\gamma_{h_{m-1}}(q_{m,   m-1}).
$$
For $q \geqslant 0$ and $R = (r_1,\ldots,r_m)$, set
$$ R_q^* = (q - 2(r_1 + \ldots + r_m),r_1,\ldots, r_{m-1}).
$$

Let $V$ be a vector space over $\mathbb Z/p$ and $V^*$ be its dual. Denote by
$$\langle.,.\rangle : V \otimes V^* \longrightarrow \mathbb Z/p$$
the dual pairing.

The main result of the section is

\begin{thm}\label{dl21} Suppose given $e, \delta = 0,  1,\, j \geqslant 0,\, (S,R)$, and $(S',R')$ with $\ell(R)  = k, \,  \ell(R') =  n,\,   \ell(S) = t  \leqslant k,\, \ell(S') = t' \leqslant n$. Set
$ \sigma =  r(S,R)  +  r(S',R')  +  s  +  \delta  +  (t  + [-2p^s])t' + nhk\delta ,$ with $-\delta \leqslant s \leqslant n - \delta$. Then we have
\begin{multline*}\langle  \tilde  m_S\tilde  q_{R^*_{(2-\delta)p^n  -  e - 2j -	t}}\otimes  u_{k+1}^e\gamma_j(v_{k+1}),   St^{S',R'}\big(U_{k+1}^\delta V_{k+1}^{1-\delta}\big)\rangle \\
	=  \begin{cases} (-1)^\sigma\langle  \tilde  m_{S'}\tilde q_{{R'}^*_{(2-\delta)p^k
			- t'}}, St^{S,R}\big(\tilde M_{n,s}^\delta Q_{n.s}^{1-\delta}\big)\rangle,
	&e + 2j = -[-2p^s],\\ 0 , &\text{otherwise.}
	\end{cases} \end{multline*} 
	Here, by convention, $\tilde M_{n,-1} = \tilde L_n.$
\end{thm}
\begin{proof} We prove the theorem for $\delta = 1$. For $\delta = 0$, it is similarly proved. We set
\begin{align*} U &=  U_{n +  k +1}(x_1,\ldots,x_k,x'_1,\ldots,x'_n,x,y_1,\ldots,
y_k,y'_1,\ldots,y'_n,y),\\ 
U' &= U_{n + k +1}(x'_1,\ldots,x'_n,x_1,\ldots,x_k,x,y'_1,\ldots,
y'_n,y_1,\ldots,y_k,y). \end{align*} 
It is easy to verify that
\begin{align*} U \ =\ (-1)^{nkh}U'\ .\tag a\end{align*}
Computing directly from Proposition \ref{md11} gives
 \begin{align*} U &=   (-h!)^{-k}d^*_kP_kU_{n+1}(x'_1,\ldots,x'_n,x,y'_1,\ldots,
y'_n,y)\tag b\\ 
&=   (-h!)^{-k}(-1)^nd_k^*P_k\big(\sum_{s   =  -1}^{n  -  1}
(-1)^{s+1}\tilde M_{n,s}y^{p^s}\big)\\ 
&= (-1)^n \sum_{s=-1}^{n-1} (-h!)^{-(s+1)k/(|s|
	+ 1)}(-1)^{s+1}(d_k^*P_k\tilde M_{n,s})V_{k+1}^{p^s}\ .
\end{align*} 
Here by convention $y^{1/p} = x$, and $V_{k+1}^{1/p} = U_{k+1}$.

We observe that $\dim\tilde M_{n,s} =  p^n  +  [-2p^s]$. According to Theorem \ref{dl12} we have 
\begin{align*}  d^*_kP_k\tilde M_{n,s} = \mu(p^n +
[-2p^s])^k\sum_{S,R} (-1)^{r(S,R)}\tilde M_S\tilde  Q^{R^*_{p^n+[-2p^s]-t}}St^{S,R}\tilde M_{n,s} .\tag c\end{align*} 
A simple computation shows that
\begin{align*} (-h!)^{ -(s+1)/(|s| +  1)}\mu(p^n +\  [-2p^s]) =  (-1)^{nh} .
\tag d\end{align*} 
Combining (b),(c) and (d) we get
$$ U = \sum_{s=-1}^{n-1}\Big(\sum_{S,R} (-1)^{n(kh+1)+r(S,R)+s+1}\tilde M_S
\tilde  Q^{R^*_{p^n   +   [-2p^s]-t}}   St^{S,R}\tilde   M_{n,
	s}\Big)V_{k+1}^{p^s} .$$ 
From this, we see that it implies
\begin{align*}  &(-1)^{r(S,R)+n(hk+1)+s+1}\langle \tilde m_S\tilde  q_{R^*_{p^n - e - 2j - t}} \otimes   \tilde   m_{S'}\tilde   q_{{R'}^*_{p^k-t'}}\otimes u_{k+1}^e\gamma_j(v_{k+1}), U \rangle\tag e \\ 
&\qquad  = \begin{cases}  (-1)^{tt'}\langle  \tilde  m_{S'}\tilde q_{{R'}^*_{p^k -
		t'}}, St^{S,R}(\tilde M_{n,s})\rangle, &e + 2j = -[-2p^s],\\
0\ , &\text{otherwise.}\end{cases}\end{align*} 

On the other hand, from Proposition \ref{md11} and Theorem \ref{dl12} we have
\begin{align*} U' &=  (-h!)^{-n}d_n^*P_nU_{k+1}(x_1,\ldots,x_k,x,y_1,\ldots,y_k,
y)\\ 
&=(-h!)^{-n}\mu(p^k)^n \sum_{S',R'} (-1)^{r(S',R')}
\tilde M_{S'}\tilde Q^{{R'}^*_{p^k-t'}}St^{S',R'}U_{k+1}.
\end{align*}  
From this and the fact that $(-h!)^{-1}\mu(p^k) = (-1)^{hk}$, we get
\begin{align*} &(-1)^{r(S',R') + n(hk + 1)}\langle  \tilde m_S\tilde q_{R^*_{p^n
		- e - 2j - t}} \otimes \tilde m_{S'}\tilde q_{{R'}^*_{p^k-t'}}
\otimes u_{k+1}^e\gamma_j(v_{k+1}), U' \rangle\tag f \\
&\quad = (-1)^{t'e}\langle \tilde m_S\tilde q_{R^*_{p^n  - e - 2j
		- t}} \otimes u_{k+1}^e\gamma_j(v_{k+1}),   St^{S',  R'}U_{k+1}
\rangle . \end{align*} 
Comparing (e) with (f) and using (a), we obtain the theorem for $\delta = 1$.
\end{proof}

Since the basis $\{\tilde M_{S'}\tilde Q^{H'}\}$ of $E(\tilde M_{n,0},\ldots , \tilde M_{n,n-1})\otimes P(\tilde L_n,Q_{n,1}, \ldots,Q_{n,n-1})$ is dual
to the basis $\{\tilde m_{S'}\tilde q^{H'}\}$ of $E(\tilde m_{n,0},\ldots , \tilde m_{n,n-1})\otimes \Gamma(\tilde \ell_n,q_{n,1}, \ldots,q_{n,n-1})$. Hence, we easily
obtain from Theorem \ref{dl21}
\begin{corl}\label{hq22} Set 
$$ C_{S',R'} = \langle \tilde m_S\tilde q_{R^*_{(2-\delta)p^n + [-2p^s] - t}}\otimes \gamma_{p^s}(v_{k+1}), St^{S',R'}\big(U_{k+1}^\delta V_{k+1}^{1-\delta}\big)\rangle. $$
	We have
	$$ St^{S,R}\big(\tilde M_{n,s}^\delta Q_{n,s}^{1-\delta}\big) =
	\sum_{S' ,R'}   (-1)^\sigma   C_{S',R'}\tilde
	M_{S'}\tilde Q^{{R'}^*_{(2-\delta)p^k-t'}}.$$ 
Here, by convention, $\gamma_{1/p}(v_{k+1}) = u_{k+1}.$ 
\end{corl}

By an analogous argument we obtain

\begin{corl}\label{hq23} Set 
	$ C_{s,S,R} = \langle \tilde m_{S'}\tilde q_{{R'}^*_{(2-\delta)p^k - t'}},
	St^{S,R}\big(\tilde M_{n,s}^\delta Q_{n,s}^{1-\delta}\big)\rangle. $
	We have
	$$ St^{S',R'}\big(\tilde U_{k+1}^\delta V_{k+1}^{1-\delta}\big) =
	\sum_{s=-\delta}^{n-\delta}\Big( \sum_{S,R} (-1)^\sigma C_{s,S,R}
	\tilde M_S \tilde Q^{R^*_{(2-\delta)p^n+[-2p^s]-t}}\Big)V_{k+1}^{p^s}\ ,$$
	Here, by convention, $V_{k+1}^{1/p} = U_{k+1}$ .
\end{corl}
\section{Applications}\label{s4}

Fix a non-negative integer $r$. Let $\alpha_i = \alpha_i(r)$ denote the $i$-th coefficient in $p$-adic expansion of $r$. That means
$$ r = \alpha_0p^0 + \alpha_1p^1 +\ldots$$
with $0 \leqslant \alpha_i < p,\, i \geqslant 0$. Set $\alpha_i = 0$ for $i < 0$.

The aim of the section is to prove the following four theorems:

\begin{thm}\label{dl31} Set  
$c = \frac{(h-1)!}{(h-\alpha_{k-1})!\prod_{0\leqslant i<k} 
(\alpha_i-\alpha_{i-1})!},   \ t_i = \alpha_i-\alpha_{i-1},\ 0 \leqslant i < k$. We have 
\begin{align*} P^rU_{k + 1} &= \begin{cases}\displaystyle{
	c\Big( hU_{k+1} + \sum_{u=0}^{k-1}t_uV_{k+1}\tilde M_{k,u}Q_{k,u}^{-1}\Big)
	\prod_{i=0}^{k-1}Q_{k,i}^{t_i}},\ &2r < p^k ,\,  t_i
	\geqslant 0,\ i < k ,\\ 
	0\ ,  &\text{otherwise.}\end{cases}
	\end{align*}
\end{thm}
\begin{thm}\label{dl32} Set $c = \frac{(h-\alpha_s)(h-1)!}{(h-\alpha_{n-1})!
		(\alpha_s+1-\alpha_{s-1})!\prod_{s\not= i<n} (\alpha_i-\alpha_{i-1})!},
	t_i = \alpha_i-\alpha_{i-1},\ -1 \leqslant i \ne s,\
	t_s = \alpha_s+1-\alpha_{s-1}$,
	with $-1 \leqslant s \leqslant n - 1$. We have
	\begin{align*} P^r\tilde M_{n,s} &=
	\begin{cases}  \displaystyle{c\sum_{u=-1}^st_u\tilde M_{n,u}
	Q_{n,u}^{t_u-1}\prod_{u\not= i<n} Q_{n,i}^{t_i}} ,
	&2r\leqslant p^n+[-2p^s], \alpha_i \geqslant \alpha_{i-1},\\
	&s \not= i < n,\ \alpha_s+1\geqslant \alpha_{s-1}\ ,\\
	0\ ,&\text{otherwise.}\end{cases}
	\end{align*}   
\end{thm}

The following two theorems were first proved in \cite{hum} by another method.

\begin{thm}[H\uhorn ng-Minh \cite{hum}]\label{dl33} 
\begin{align*}P^rV_{k+1} = \begin{cases} V_{k+1}^p, &r = p^k,\\
	\frac{(-1)^{\alpha_{k-1}}\alpha_{k-1}!}{\prod_{\scriptstyle 0
			\leqslant i<k}(\alpha _i-\alpha _{i-1})!}\displaystyle{V_{k+1}\prod_{i=0}^{k-1}
	Q_{k,i}^{\alpha  _i-\alpha  _{i-1}}},&
	r < p^k , \ \alpha _i \geqslant \alpha _{i-1},\ i <k,\\
	0\ , &\text{otherwise.}\end{cases}
\end{align*}
\end{thm}

\begin{thm}[H\uhorn ng-Minh \cite{hum}]\label{dl34} Set
	$c = \frac{(-1)^{\alpha  _{n-1}}\alpha_{n-1}!(\alpha
		_s+1)}{(\alpha  _s+1-\alpha  _{s-1})!\prod_{s\ne i <n} (\alpha  _i-\alpha  _{i-1})!}.$ Then
\begin{align*}P^rQ_{n,s} = \begin{cases} Q_{n,s}^p, &r = p^n-p^s,\\
	\displaystyle{cQ_{n,s} \prod_{0 \leqslant i <n}
	Q_{n,i}^{\alpha_i-\alpha_{i-1}}},  &r < p^n-p^s,   \alpha_i\geqslant  \alpha
	_{i-1},\\
	& s\ne  i <n,\ \alpha_s+1\geqslant \alpha_{s-1},\\
	0\ , &\text{otherwise.}\end{cases} 
\end{align*}
\end{thm}
To prove these theorems we need

\begin{nota}\label{kh35} Let $R = (r_1, \ldots , r_n)$ be a sequence of arbitrary integers and $b \geqslant 0$.
Denote by $|R| = \sum_{i=1}^n(p^i-1)r_i$ and $\binom bR$ the coefficient of $y_1^{r_1}\ldots y_n^{r_n}$ in $(1+y_1+\ldots y_n)^b$. That means,
$$\binom bR = \begin{cases} \dfrac{b!}{(b-r_1-\ldots-r_n)!r_1!\ldots r_n!},
&r_1 + \ldots + r_n \leqslant b,\\
0,   &r_1 + \ldots + r_n > b.
\end{cases}$$
\end{nota}

The proofs of Theorems \ref{dl31} and \ref{dl33} are based on the duality theorem and the following
\begin{lem}\label{bd36} Let $b$ be a non-negative integer and $\varepsilon = 0,\, 1$. We then have
	\begin{align*} 
St^{S,R}(x^\varepsilon y^b) = \begin{cases} 
\displaystyle{\binom bR x^\varepsilon y^{b + |R|} ,} &S = \emptyset ,\\
\displaystyle{\varepsilon \binom bR y^{b +  |R| + p^u}} , &S = (u), \ u \geqslant 0 ,\\ 
	0\ , &\text{otherwise.} \end{cases}
\end{align*}
Here $x$ and $y$ are the generators of $H^*(B\mathbb Z/p) = E(x)\otimes P(y)$.
\end{lem}
\begin{proof} A direct computation using Proposition \ref{md11} shows that
\begin{align*} d_m^*P_m(x^\varepsilon y^b) &= (-1)^{mb}(h!)^{m
	\varepsilon}\Big(\tilde L_m^\varepsilon
x^\varepsilon  + \varepsilon \sum_{u=0}^{m-1}
(-1)^{u+1}\tilde M_{m,u}y^{p^u}\Big)\\
&\hskip 3 cm   \times\Big(\sum_{R=(r_1,\ldots,
	r_m)}(-1)^{r(\emptyset ,R)}\binom bR 
\tilde Q^{R^*_{2b}}y^{b + |R|}\Big)\\
&= \mu(2b+\varepsilon)^m\Big(\sum_{R=(r_1,\ldots,r_m)}(-1)^{r(\emptyset ,R)}
\binom bR\tilde Q^{R^*_{2b+\varepsilon}}x^\varepsilon y^{b + |R|}\\ 
&\hskip 1 cm + \varepsilon\sum_{u=0}^{m-1}\sum_{R=(r_1,\ldots,r_m)}
(-1)^{r((u),R)}\binom bR\tilde M_{m,u}
\tilde Q^{R^*_{2b}}y^{b+|R| + p^u}\Big). 
\end{align*} 
The lemma now follows from Theorem \ref{dl12}.
\end{proof}

\begin{proof}[Proof of Theorem \ref{dl31}] Since $\dim U_{k+1} = p^k$, it is clear that$P^rU_{k+1} = 0$ for $ 2r > p^k$.
Suppose $r \leqslant (p^k-1)/2$. Applying Corollary \ref{hq23} with $\delta = n = 1$ and using Lemma \ref{bd36} we obtain
\begin{align*} P^rU_{k+1}   &=   \sum_{R   =   (r_1,\ldots,
		r_k)}(-1)^{r(\emptyset,R)+r +hk} 
	\langle\tilde   q_{(r)^*_{p^k}},St^{\emptyset,R}\tilde   L_1\rangle
	U_{k+1}\tilde Q^{R^*_{p-1}}\\ 
	&\quad + \sum_{u=0}^{k-1}\sum_R(-1)^{r((u),R)+r+kh+1}
	\langle\tilde q_{(r)^*_{p^k}},St^{(u),R}\tilde M_{1,0}
	\rangle V_{k+1}\tilde M_{k,u}\tilde Q^{R^*_{p-3}}\ .
	\end{align*} 
Set  $\bar r_i = \alpha_i-\alpha_{i-1},\ i < k,\ \bar r_k
= h - \alpha_k,\ \bar R_0 = (\bar r_1,\ldots,\bar r_k),
\bar  R_u = (\bar r_1,\ldots,\bar r_u-1,\ldots,\bar r_k),\ 1\leqslant u \leqslant k.$ Computing directly from Lemma \ref{bd36} with $\varepsilon = 0,\  b = h$ or $\varepsilon = 1,\ b = h - 1$  gives
\begin{align*}   \langle\tilde   q_{(r)^*_{p^k}},   St^{\emptyset,
	R}\tilde L_1\rangle &= \begin{cases}  \dfrac{h!}{\bar r_0\ldots\bar r_k},\quad  &R
= \bar R_0 \\ 
0\ , &\text{otherwise.}\end{cases}\\
\langle\tilde q_{(r)^*_{p^k}},St^{(u),R}\tilde M_{1,0}\rangle &=
\begin{cases}  \dfrac{(h-1)!\bar r_u}{\bar r_0\ldots\bar r_k},\quad &R  = \bar R_u\\
0\ , &\text{otherwise.}\end{cases} 
\end{align*}
A simple computation shows that
$$r( \emptyset,\bar   R_0)  +  r  =  r((u),\bar  R_u)  +  r  +  1  =  hk\
(\text{mod}\ 2).$$ 
Hence, the theorem is proved.
\end{proof}

\begin{proof}[Proof of Theorem \ref{dl33}] Since $\dim V_{k+1} = 2p^k$, we have only to prove the theorem for $r < p^k$. Note that $Q_{1,1} = 1$. Hence
	$$ St^{S,R}Q_{1,1} = \begin{cases}  1\ , \quad &S = \emptyset,\, R = (0,\ldots,0),\\
	0,\  & \mbox{otherwise}.
	\end{cases}$$
So, $\langle  \tilde q_{(r)^*_{2p^k}}, St^{S,R}Q_{1,1}\rangle = 0$  for any $S,\, R$. Remember that $Q_{1,0} = y^{p-1}$. So, applying Corollary \ref{hq23} with $\delta = 0,\, n = 1$ and using Lemma \ref{bd36} with $\varepsilon  = 0,\, b = p - 1$, we get
\begin{align*} P^rV_{k+1} = \sum_{R} (-1)^{r(\emptyset ,R) + r}\langle \tilde q_{(r)^*_{2p^k}},   St^{\emptyset,R}Q_{1,0}\rangle  \tilde Q^{R^*_{2(p-1)}}V_{k+1}.\tag a
\end{align*}
From Lemma \ref{bd36}, we see that it implies
\begin{align*} &\langle \tilde q_{(r)^*_{2p^k}},St^{\emptyset,R}Q_{1,0}\rangle\tag b\\
&\qquad = \begin{cases} \displaystyle{\binom{p-1}R} ,  &R = (\alpha_1 - \alpha_0,\ldots,
\alpha_{k-1}-\alpha_{k-2},p-1-\alpha_{k-1}),\\
0\ , &\text{otherwise.}\end{cases} 
\end{align*}

Suppose that $R = (\alpha_1 - \alpha_0,\ldots,\alpha_{k-1}-\alpha_{k-2},p-1-\alpha_{k-1}).$ Then we can easily observe that
\begin{align*} &r(\emptyset ,R) + r = 0\ (\text{mod}\ 2),\tag c\\
	&\binom{p-1}R = \frac{(-1)^{\alpha_{k-1}}\alpha_{k-1}!}
	{\prod_{0\leqslant i < k}(\alpha_i-\alpha_{i-1})},\\
	&R^*_{2(p-1)} = (2\alpha_0,\alpha_1-\alpha_0,\ldots,
	\alpha_{k-1}-\alpha_{k-2}). 
\end{align*}
Theorem \ref{dl33} now follows from (a),(b) and (c).
\end{proof}

Following Corollary \ref{hq22},to determine $P^rM_{n,s}$ and $P^rQ_{n,s}$ we need to compute the action of $St^{S,R}$ on $U_2$ and $V_2$.
\begin{prop}\label{md37} Suppose given $R = (r_1,\ldots,r_n)$, and $0\leqslant u < n$. Set $w_{s} = r_{s+1} + \ldots + r_n$, for $ s \geqslant 0$.  Then we have
$$ St^{S,R}U_2 = \begin{cases} \displaystyle{\binom hR\big(\tilde L_1^{\vert
	R\vert /h}U_2 + \overset{n-1}{\underset{s=0}\sum}
h^{-1}w_{0}\tilde M_{1,0}\tilde     L_1^{(\vert R\vert -
	p^{s+1}+1)/h}V_2^{p^s}\big),\quad S = \emptyset},\\
\displaystyle{\binom hR\overset{n-1}{\underset{s=u}\sum}
h^{-1} w_{s}\tilde  L_1^{(\vert R\vert  - p^{s+1}+p^u+h)/h}V_2^{p^s},
\hskip2.6cm S = (u)},\\
0\ ,\hskip 6cm \text{otherwise.}\end{cases}$$
Here, $|R|$ and $\binom hR$ are defined in Notation \ref{kh35}.
\end{prop}
The proposition will be proved by using Theorem \ref{dl12} and the following
\begin{lem}\label{bd38} Let $u, v$ be non-negative integers with $u \leqslant v$. We have
	
\medskip
{\rm i)} $[u,v] = \sum_{s=u}^{v-1}V_1^{p^v}-p^{s+1}+p^uV_2^{p^s}$,
	
{\rm ii)} $[1;v]  =   V_1^{p^v-h}U_2   +	M_{1,0}\sum_{s=0}^{v-1}V_1^{p^v-p^{s+1}}V_2^{p^s}. $

\medskip
Here $[u,v]$ and $[1;v]$ are defined in Section \ref{s2}. 
\end{lem}

The proof is straightforward.

\begin{proof}[Proof of Proposition \ref{md37}] Recall that $ M_{2,1} =
x_1y_2 - x_2y_1$. From Proposition \ref{md11} we directly obtain
$$d_n^*P_nM_{2,1} = (-h!)^n \sum_{v=0}^n(-1)^v\tilde L_nQ_{n,v}[1;v] +
\sum_{\scriptstyle 0\leqslant u < n\atop\scriptstyle 0\leqslant v\leqslant n}
(-1)^{u+v+1}\tilde M_{n,u}Q_{n,v}[u,v].$$
Since $L_1 = y_1$ and $2(h - 1) = p - 3$, using Proposition \ref{md11}(iii) with $y = y_1$ and Notation \ref{kh35} we get
$$d_n^*P_nL_1^{h-1} = (-1)^{n(h-1)}\sum_{R'}(-1)^{r(\emptyset,R')}
\binom{h-1}{R'}\tilde Q^{{R'}^*_{p-3}}y_1^{\vert R'\vert +h-1}.$$
We have $U_2 = M_{2,1}L_1^{h-1},\ \dim U_2 = p$ and $\mu(p) = (-1)^hh!$. So, it implies from the above equalities and Proposition \ref{md11} that
\begin{align*} d_n^*P_n   U_2  &=  \mu(p)^n\Big(\sum_R(-1)^{r(\emptyset,R)}\tilde
Q^{R^*_p}\binom  hR \sum_{v=0}^nh^{-1}r_vy_1^{\vert  R\vert +  h - p^v}[1 ;
v]\\   &+\sum_{u=0}^{   n-1}   \sum_R(-1)^{r((u),R)}\tilde   M_{n,u}\tilde
Q^{R^*_{p-1}}\binom hR \sum_{v=u}^nh^{-1}r_vy_1^{\vert R\vert +  h - p^v}[u; v]\Big). \end{align*} 
Then by Theorem \ref{dl12} we have
$$ St^{S,R}U_2 = \begin{cases}\displaystyle{ h^{-1}\binom hR\sum_{v=0}^nr_vy_1^{\vert
	R\vert + h  - p^v}[1 ; v]},\ &S = \emptyset,\\
\displaystyle{h^{-1} \binom hR\sum_{v=u}^nr_vy_1^{\vert R\vert + h  - p^v}[u ;  v]},\ &S =
(u),\ u < n,\\ 
0\ , &\text{otherwise.}\end{cases}$$
Now the proposition follows from Lemma \ref{bd38}.
\end{proof}

\begin{proof}[Proof of Theorem \ref{dl32}] For simplicity, we assume that $0 \leqslant s < n$. Applying Corollary \ref{hq22} with $\delta = k = 1$ and using Proposition \ref{md37} we get
\begin{align*}P^r\tilde  M_{n,s}  =  \sum_{\scriptstyle  0\leqslant  u\leqslant s\atop\scriptstyle R = (r_1,\ldots,r_n)} 
(-1)^{r((u),R)+r+s+1+nh}C_{(u),R}\tilde  M_{n,u}\tilde Q^{R^*_{p-1}}.
\end{align*} 
Here $C_{(u),R} = \langle \tilde q_{(r)^*_{p^n-2p^s}}\otimes \gamma_{p^s}(\bar v_2) , St^{(u),R}U_2\rangle .$

If $2r > p^n - 2p^s - 1$, then $P^r\tilde  M_{n,s} =  0$ since d$\dim \tilde M_{n,s} = p^n - 2p^s$. Suppose $2r  \leqslant p^n  - 2p^s - 1$. Set $t_i = \alpha_i-\alpha_{i-1}, $ with $ 0 \leqslant i \ne  s,\,  n, \ t_s = \alpha_s + 1 -  \alpha_{s-1},\ t_n  = h  - \alpha_{n-1},\  \bar R_0 = (t_1,\ldots,t_n),\  \bar  R_u  =  (t_1,\ldots,t_u - 1,\ldots,
t_n),\ 1 \leqslant u \leqslant n.$ From Proposition \ref{md37} we have
$$ C_{(u),R} = \begin{cases} ct_u\ ,\ &R = \bar R_u,\\
0\ ,&\text{otherwise.}\end{cases}$$
It is easy to verify that
$$   r((u),\bar R_u) + r + s + 1 = nh\ (\text{mod}\ 2).$$
Theorem \ref{dl32} now is proved by combining the above equalities.
\end{proof}

Now we return to the proof of Theorem \ref{dl34}. It is proved by the same argument as given in the proof of Theorem \ref{dl32}. We only compute $St^{S,R}V_2$.

\begin{prop} For $R = (r_1,\ldots, r_n)$, $r_0  = p -  r_1 - \ldots - r_n $, we have
	$$St^RV_2 = \begin{cases} V_2^{p^s}, \hskip 2cm r_s = p,\ r_i = 0,\ i \ne \ s,\\
	\displaystyle{\sum_{s=0}^{n-1}
	\frac{(p-1)!(r_{s+1}+\ldots+r_n)}{r_0!\ldots r_n!}
	V_1^{\vert  R\vert + p-p^{s+1}}V_2^{p^s}},\\
	\hskip 2.8cm  0\leqslant  r_i  < p,\ 0\le
	i\leqslant n, \\
	0\ , \hskip 2.3cm\text{otherwise.}\end{cases}$$
\end{prop}
\begin{proof}
Recall that $V_2 = y_2^p - y_2y_1^{p-1}$. Applying Proposition \ref{md11} and Lemma \ref{bd36} with $y = y_1$ or $y = y_2$ we get
 \begin{align*} d_n^*P_nV_2 &= \sum_{s=0}^n (-1)^{n + s}
Q_{n,s}^py_2^{p^{s+1}}\tag a\\
&\quad - (-1)^n\sum_{u=0}^n\sum_{R'}(-1)^{u+r(\emptyset ,R')}
\binom{p-1}{R'}Q_{n,u}  \tilde  Q^{{R'}^*_{2(p-1)}}y_1^{\vert  R'\vert  + p - 1}y_2^{p^u}\\ 
&= (-1)^n\sum_{s=0}^n (-1)^s Q_{n,s}^p\big(y_2^{p^{s+1}}
-  y_2^{p^s}y_1^{(p-1)p^s}\big)\\
&\quad - (-1)^n\sum_{u=0}^n\sum_R(-1)^{r(\emptyset ,R)}\binom{p-1}{R_u}
\tilde Q^{R^*_{2p}}y_1^{\vert R\vert + p - p^u}y_2^{p^u}\Big).
\end{align*} 
Here the last sum runs over all $ R =  (r_1,\ldots, r_n)$ with $ 0 \leqslant r_i < p,   \ 0 \leqslant i \leqslant n,  \ R_0 = R,   \ R_u = (r_1,
\ldots, r_u-1,\ldots,r_n),\ 1 \leqslant u \leqslant n$.

Let $v$ be the greatest index such that $r_v > 0$. A simple computation shows
\begin{align*}	y_1^{|R| +  p  -  p^u}y_2^{p^u}= - y_1^{|R| + p - p^u  -  p^v}[u  ,   v]    +  y_1^{|R| + p - p^v}y_2^{p^v}.\tag b
\end{align*}
Combining (a), (b), Lemma \ref{bd38} and the fact that $\sum_{u=0}^n\binom{p-1}{R_u}  =  0$ we obtain
\begin{multline*} d_n^*P_nV_2 =
\mu(2p)^n\Big(\sum_{s=0}^n(-1)^sQ_{n,s}V_2^{p^s}   \\   +
\sum_R(-1)^{r(\emptyset ,R)}\tilde Q^{R^*_{2p}}\sum_{s=0}^n\sum_{u=s+1}^n\binom{p-1}{R_u}
V_1^{\vert R\vert+p-p^{s+1}}V_2^{p^s}\Big).
\end{multline*}
The proposition now follows from this equality and Theorem \ref{dl12}.
\end{proof}

\bigskip
{}

\vskip.6cm
\hskip1cm Department of Mathematics

\hskip1cm Quinhon Pedagogic University

\hskip1cm  Dai hoc Su Pham Quinhon

\hskip1cm 170 Nguyen Hue Quinhon, Vietnam

\end{document}